\documentclass[12pt]{amsart}

\usepackage[margin=1.15in]{geometry}

\usepackage{amsmath,amscd,amssymb,amsfonts,latexsym,wasysym, mathrsfs,mathtools,hhline,color, bm}
\usepackage[all, cmtip]{xy}
\usepackage{comment}

\usepackage{url}

\definecolor{hot}{RGB}{65,105,225}

\usepackage[pagebackref=true,colorlinks=true, linkcolor=hot ,  citecolor=hot, urlcolor=hot]{hyperref}%make all the references and links clickable

\usepackage{ textcomp }
\usepackage{ tipa }

\usepackage{graphicx,enumerate}

\usepackage{enumitem}

\theoremstyle{plain}
\newtheorem{theorem}{Theorem}[section]
\newtheorem{proposition}[theorem]{Proposition}

\newtheorem{lemma}[theorem]{Lemma}

\theoremstyle{definition}

\newtheorem{definition}[theorem]{Definition}

\newtheorem{remark}[theorem]{Remark}

\newtheorem{example}[theorem]{Example}
\newtheorem*{ex*}{Example}

\numberwithin{equation}{section}

\newcommand\sS{\mathscr{S}}

\newcommand\e{{\varepsilon}}
\newcommand\m{{\setminus}}

\DeclareMathOperator{\reg}{reg}                  % reg
                  
\DeclareMathOperator{\Eu}{\mathrm{Eu}}

\newcommand{\Sing}{{\mathrm{Sing}}}%\hspace{1pt}

\DeclareMathOperator{\mult}{mult}  
\newcommand{\lk}{{\mathbb C \rm{lk}}}

\DeclareMathOperator{\grad}{grad}

\DeclareMathOperator{\ord}{ord}

\newcommand{\fin}{\hspace*{\fill}$\square$\vspace*{2mm}}
\renewcommand{\d}{{\mathrm{d}}}

\def\bP{\mathbb{P}}
\def\bC{\mathbb{C}}
\def\sS{\mathscr{S}}

\def\bZ{\mathbb{Z}}

\title[Morse numbers of function germs]{Morse numbers of function germs with isolated singularities}

\author{Lauren\c{t}iu Maxim}
\address{L. Maxim: Department of Mathematics,  University of Wisconsin-Madison,  480 Lincoln Drive, Madison WI 53706-1388, USA.}
\email {maxim@math.wisc.edu}

\author[M. Tib\u{a}r]{Mihai Tib\u{a}r}
\address{M. Tib\u{a}r: Universit\' e de  Lille, CNRS, UMR 8524 -- Laboratoire Paul Painlev\'e, F-59000 Lille, France}  
\email {mihai-marius.tibar@univ-lille.fr}


\thanks{L. Maxim was partially supported by the Simons Foundation (Collaboration Grants for Mathematicians \#567077), CNRS and the MPIM-Bonn. M. Tib\u{a}r acknowledges partial support from the Labex CEMPI grant (ANR-11-LABX-0007-01) and from the GDR du CNRS ``Singularit\' es et applications''.}

\date{\today}

\begin{document}

\begin{abstract}  
A set of  Morse numbers is associated to a holomorphic function germ with stratified isolated singularity, extending  the classical Milnor number to the setting of a singular base space. 
\end{abstract}

\maketitle

\section{Introduction}\label{intro}

The topology of holomorphic function germs $g\colon (\bC^{n}, 0) \to (\bC, 0)$, $n\ge 2$, with an isolated singularity
has been studied beginning with Milnor's fundamental lecture notes \cite{Mi} in the late 1960s.  Milnor showed at that time the existence of a locally trivial fibration
\begin{equation}\label{eq:milnorfib}
  g_{|} : B_{\e}\cap g^{-1}(D_{\delta}^{*}) \to D_{\delta}^{*}
\end{equation}
where $B_{\e}\subset \bC^{n}$ is the open ball at 0 of radius $\e>0$ and $D_{\delta}^{*}\subset \bC$ is the punctured open disk at 0 of radius $\delta>0$, and where one first chooses $\e$ small enough such that all spheres of smaller radii are transversal to $g^{-1}(0)$ and then $0<\delta \ll \e$ such that all fibres of $g$ over $D_{\delta}^{*}$ are transversal to $\partial B_{\e}$. Milnor showed that the fibration \eqref{eq:milnorfib} is independent on the choices of $\e$ and $\delta$ up to isotopy, and that its 
fibre (called \emph{Milnor fibre}) is homotopy equivalent to a bouquet of spheres $S^{n-1}$. The number $\mu(g)$ of these spheres is a local invariant, known as the \emph{Milnor number}, and has several other avatars:

\begin{enumerate}
\item the codimension of the Jacobian ideal, i.e. $\dim_{\bC} {\bC\{x_{1}, \ldots , x_{n}\} }/({\frac{\partial g}{\partial x_{1}}, \ldots , \frac{\partial g}{\partial x_{n}}})$;
\item  the number of Morse points in a Morsification of $g$;
\item  the Poincar\'e-Hopf index of $\overline{\grad g}$ at $x_0$;
\end{enumerate}

\smallskip

If the source space is a singular analytic space germ $(X, 0)\subset (\bC^{N},0)$, and $f\colon (X, 0)\to (\bC,0)$ a holomorphic function germ with a stratified isolated singularity, then several extensions of the Milnor number have been defined. %%%%%%%%%%%%%%%%%%
Let us briefly recall several notions which occur when comparing some of these extensions. 

The Milnor number introduced to Goryunov \cite{Go} for
  functions on curve singularities $X \subset \bC^N$, let us denote it by $\mu_G(f)$, is preserved under simultaneous deformations  both of the space $X$ and of the function $f$, see \cite[p.178]{MS}. Assuming that the curve singularity $(X,0)\subset (\bC^{p+1},0)$ is an ICIS defined by $g\colon (\bC^{p+1}, 0) \to  (\bC^p, 0)$, and that $F$ is an extension of $f$
 to $(\bC^{p+1}, 0)$, it follows that $\mu_G(f)$ counts the number of critical
 points with multiplicity of the restriction of $F$ to a Milnor fibre $g^{-1}(t)$. 
 This makes sense for higher dimensional ICIS too, and it turns out  that $\mu_G(f)$ equals the 
 Poincar\'e-Hopf index of the gradient of the restriction $F_{|_{X_t}}$ to the Milnor fibre $X_t$ of the ICIS $X$. 
 Consequently,  $\mu_G(f)$ also identifies to the so-called GSV-index of the gradient vector field of $f$ on $X$, introduced in \cite{GSV},  and defined as the Poincar\'e-Hopf index of an extension of the gradient of $f$ to the Milnor fibre
 $g^{-1}(t) = X_{t}$.
 %%%%%%%%%%%%%%%%
 For an ICIS, the invariant $\mu_G(f)$ also equals  the 
{\em virtual multiplicity} at $x_0$ of the function $f$ on $X$ introduced by Izawa and Suwa 
\cite{IS}. This multiplicity is by definition the localisation at $0$ of the top Chern class of the 
virtual cotangent bundle $T^*(X)$ of $X$  defined by  the differential of $f$, which is non-zero on $X \m \{0\}$ by 
hypothesis. The virtual multiplicity has the advantage of being defined even if the singular set of $X$ is non-isolated, and we refer to  \cite{IS} for details. 
In case of an ICIS, the invariant $\mu_G(f)$ also coincides with the index of the 1-form $dg$ defined in \cite{EG}. This turns out to be  close to the interpretation  given in \cite{LSS} of the GSV index as a  localisation  of the top Chern class of the virtual tangent bundle. 

%%%%%%%%%%%%%%%%%%%%%

\
 
 One may remark that all the invariants evoked above are related to items (a) and (c) in the list of Milnor number avatars.
 In contrast, we are concerned here with item (b). 

 In the context of a holomorphic function germ $g\colon (\bC^{n}, 0) \to (\bC, 0)$, $n\ge 2$, with an isolated singularity, 
 Brieskorn  showed in \cite[Appendix]{Bri} that any continuous family $g_{t}$ of holomorphic Morse function germs with $g_{0} = g$  has  precisely the number $\mu(g)$ of Morse points which converge to $0$ as $t\to 0$. 
 
Let us consider holomorphic functions $f\colon (X,0) \to (\bC,0)$ on a singular  stratified space germ $(X,0)$. 
The space $X$ may be endowed,   in some small enough neighborhood of $0\in X$, with the coarsest Whitney stratification $\sS$ having finitely many strata,  such that $\{0\}$ is the unique point-stratum.  We introduce here, or rather recall (see our comments below) the set $M(f) := \{ (V,m_{V}) \mid V\subset \sS \setminus \{0\} \}$ of \emph{Morse pairs}  associated to such $f$, where  for any positive dimensional stratum $V\in \sS$, the integer  $m_{V}$ denotes the number of Morse points of $f_{t}$ on $V$, for some Morsification $f_{t}$ of $f := f_{0}$ and small enough parameter $t\in \bC$.
We observe in \S\ref{ss:mainmorsedef} that $m_{V}$ is independent on the chosen Morsification, and therefore that these Morse numbers are topological invariants of $f$.  We also remark in Proposition \ref{p:MorseSTV} that
\cite{STV2} gives the interpretation of the Morse numbers in terms of the relative Euler obstruction (defined in \cite{BMPS}), namely one has:
\begin{equation}
 m_{V} = (-1)^{\dim V} \Eu_{f}(\overline{V},0).
\end{equation}
  
%The main scope of this note is to point out the most practical and effective  formula for the bouquet of Morse numbers.
%which is in terms of two \emph{local polar multiplicities}.
The computability of the Morse numbers is an important issue for the vanishing homology of $f$. Indeed, the numbers $m_{V}$
occur in the following direct sum decomposition of the $\bZ$-homology of the Milnor fibre $F_{f}$ of $f$, as important ingredients together with the complex links $\lk_{X}V$ of the strata $V\in \sS$, which follows from the Brieskorn deformation principle (see \S\ref{ss:bri}):
\[ \widetilde H_{*}(F_{f}) \cong  \widetilde H_{*}(\lk_{X}\{0\}) 
  \oplus  \bigoplus_{V\in \mathscr S\setminus \{0\}}   
  % \bigoplus_{k_{0,V}(f) {\rm \tiny{\hspace{1pt} times}}} 
    \widetilde H_{*}\left(\Sigma^{\dim V}(\lk_{X}V) \right)^{m_{V}}.
\]

The main scope of this note is to point out a most practical and effective way of computing the set of Morse numbers
in terms of \emph{local polar multiplicities}.
 Our  Theorem \ref{t:main1} compares the above direct sum decomposition to a different decomposition shown in \cite{Ti-bouquet} which is provided by a polar analysis of the Milnor fibre.  
 It actually identifies the two decompositions, which amounts to proving the equality of $m_{V}$ with the polar invariant $k_{V}$ occurring in \cite[\S 4.1]{Ti-bouquet}, see \eqref{eq:polar0}, for any positive dimensional stratum $V$. This identification of numbers has also been shown by  Massey as his main result \cite[Theorem 3.2]{Ma}.
What we offer here is a short and direct way of proving this polar formula for the Morse numbers $m_{V}$.
At the end, we explain on one example all the details of the algorithmic computation of the polar invariants $k_{V}$ which yield the Morse numbers.

\smallskip

\noindent \emph{Acknowledgements.} We thank Matthias Zach for bringing up the question of the computability
of Morse numbers as they appear in Massey's paper \cite{Ma}.

%%%%%%%%%
%%%%%%%%%%%%%%%%
\section{Brieskorn principle and Morse numbers}

We assume from now on that the holomorphic function germ $f \colon (X, 0)\to (\bC,0)$ has an isolated singularity with respect to the canonical Whitney stratification $\sS$ of $X\subset \bC^{N}$ at $0$. 

%%%%%%%%%%
 \subsection{Polar locus and general linear functions}
 
  Let $l\colon (X, 0)\to \bC$ be a linear function, and let  
$\Gamma(l, f) := \Sing_{\sS}(l,f)$ be the germ at the origin of the \emph{polar locus of $f$ with respect to $l$}. 

One has the following fundamental result of Bertini-Sard type which goes back to Hamm-L\^{e}, Kleiman and Teissier; we  refer to  \cite{Kl}, \cite{HL}, see also \cite{Ti-compo}, \cite[Theorem 7.1.2]{Ti-book}, for proofs and for more ample discussions:
\begin{lemma}[Polar Curve Lemma]\label{l:genpolar}  
There is a Zariski open dense subset $\Omega' :=\Omega'_{f}  \subset \check \bP^{N-1}$ such that  $\Gamma(l,f)$ is either a reduced curve germ for all $l \in \Omega'$, or is empty for all $l \in \Omega'$, and that
the hyperplane $\{ l=0\} $ is transverse to all strata of $X \setminus \{ 0\}$. 

In case the polar set is non-empty,  there exists a Zariski open subset $\Omega \subset \Omega'$ 
such  that the restriction $(l,f)_{| \Gamma (l,f)}$ is  one-to-one. 
\fin
\end{lemma}
The linear forms $l\in \Omega'$ will be called \emph{general} and those in $\Omega$, \emph{very general}. We will only use $\Omega'$ and refer to Remark \ref{r:lessgeneral} for comments on $\Omega$.  The strata of dimension 1 in $\sS$ are, by definition, components of the polar curve.

%%%%%%%%%%%%
\subsection{The Brieskorn principle on Morsifications}\label{ss:bri}

  A general deformation $f_{t}$ of $f$ at $0$ has only stratified Morse singularities on certain strata $V\in \sS$ of positive dimension, in other words the singular point $0\in \Sing_{\sS} f$ splits into a finite number of Morse stratified singularities $q_{1}, \ldots , q_{r}$ of $f_{t}$, for $t\not= 0$, on some strata $V\in \sS$ which have $0$ in their closures, see \cite{GM}. 
  If the space is singular at $0$ then the origin $q_{0}:= 0\in X$ is still a singular (but non-Morse) point of $f$, due to the singular structure of $X$ at 0,
  whereas the germ $f_{t}$ at $q_{0}$ is that of a general function.  We shall call \emph{Morsification} such  a general deformation of $f$.
    A particular  Morsification of $f$ is the general linear deformation $f_{t} = f -\lambda l$ for some $l\in \Omega'$.
  
It was observed long ago, by L\^{e} D.T. among others (see, e.g., \cite{Le-isol}) that any holomorphic function
germ on a singular space germ  $(X,0)$ has a locally trivial Milnor fibration. In our setting, $f$ has a stratified isolated  singularity at the origin.
Given a small enough Milnor ball $B_{\e}$ at $0$ for $f$,  the local Milnor fibre $F_{f}$ of $f$  is homotopy equivalent to the general fibre of the Morsification $f_{t}$ inside the same ball $B_{\e}$, and all fibres of it above a small enough disk $D_{\eta}$ are stratified-transversal to the boundary $\partial B_{\e}$. If we compute  the homology of the general fibre of $f_{t}$ then, by excision at each small Milnor ball at the Morse singularities $q_{i}$ inside $B_{\e}$, and at the central singularity $q_{0}$, we obtain the following direct sum decomposition in the reduced homology with $\bZ$ coefficients:
$$ \widetilde H_{*}(F_{f}) \cong \bigoplus_{i=0}^{r} \widetilde H_{*}(F_{f_{t}, q_{i}}),$$
where $F_{f_{t}, q_{i}}$ denotes the Milnor fibre of $f_{t}$ at $q_{i}$, for all $i=0, 1,\ldots , r$.
\medskip

Let us denote by  $\lk_X(V)$ the complex link of a stratum $V \in \mathscr S$.
Let  $\Sigma^{k}(Y)$ denote the $k$-fold suspension of some space $Y$. By convention, we set $\Sigma(\emptyset ):= S^0$ (the $0$-sphere), and $\Sigma^0(Y ):= Y$.
By Goresky-MacPherson's stratified Morse theory \cite{GM}, the homology of the Milnor fibre of a stratified Morse singularity $q_{i}\in V$, $i=1, \ldots , r$,  can be expressed in terms of the local Morse data, thus in terms of the complex link of the positive dimensional stratum $V$, as follows: 
\begin{equation}\label{eq:decomp1}
 \widetilde H_{*}(F_{f_{t}, q_{i}})  \cong \widetilde  H_{*}\left( \Sigma^{\dim V}(\lk_X(V)) \right).
\end{equation}

 Note that $F_{f_{t}, q_{0}}$ is the Milnor fibre of a general function at $0\in X$, hence it is homotopy equivalent to the complex link $\lk_{X}\{0\}$.

\medskip

The above deformation principle has been established by Brieskorn  \cite[Appendix]{Bri}  in case of  a function on a smooth space $X$. It was extended by L\^{e} D.T. \cite{Le-Dmod} on any stratified space germ  $(X,0)$, by Siersma \cite{Si} on  a space germ $(X,0)$ with isolated singularity, and by the second author \cite[pag. 228-229]{Ti-bouquet} for any singular $X$. Massey's paper \cite{Ma} adapted Brieskorn's argument to the (hyper)cohomology of the Milnor fibre with constructible sheaf coefficients on a stratified $(X,0)$. (These references are not meant to be exhaustive, as there might be other earlier ones in the literature, which we are not aware of.) 

Following \cite[pag. 228-229]{Ti-bouquet}, one gets from \eqref{eq:decomp1} the isomorphism in $\bZ$-homology:
\begin{equation}\label{eq:decomp2}
\widetilde H_{*}(F_{f}) \cong  \widetilde H_{*}(\lk_{X}\{0\}) 
  \oplus  \bigoplus_{V\in \mathscr S\setminus \{0\}} \widetilde H_{*}\left(\Sigma^{\dim V}(\lk_{X}V) \right)^{m_{V}},
\end{equation}
where $X$ is of pure dimension $\dim_{0} X \ge 2$, and where:
 \begin{equation}\label{eq:morsedef}
 m_{V} := \# \{ \mbox{stratified Morse points of  } f_{t} \mbox{ on } V \mbox{ which converge to $0$ as } t\to 0\}.
\end{equation}

%\

%%%%%%%%%%%%%%
\subsection{The set of Morse pairs}\label{ss:mainmorsedef}
It will turn out from the Comparison Argument \S\ref{ss:comp}, see Remark \ref{r:morseindep}, that the integers $m_{V}$ defined by \eqref{eq:morsedef} depend only on the stratum $V$ but not on the
chosen Morsification $f_{t}$. Alternatively, this also follows from Theorem \ref{t:main1} below.
We may therefore introduce the following.
\begin{definition}\label{d:mainmorsedef}
Let $f \colon (X, 0)\to (\bC,0)$ be a holomorphic function germ with isolated singularity with respect to the canonical Whitney stratification $\sS$ of $X$ at $0$. 
The non-negative integer $m_V$ of \eqref {eq:morsedef} is called the \emph{Morse number of $f$ associated to the stratum $V\in \sS$}.
The following set of pairs ``(stratum, Morse number)'':
$$M(f) := \{ (V,m_{V}) \mid V\in \sS \setminus \{0\} \}$$
  is called the \emph{set of Morse pairs of $f$}.
\end{definition}

Let us remark that no Morse number is defined for the stratum $\{0\}\in \sS$, by the following reasons: either  (1). the point $0$  does not belong to a higher dimensional stratum in any Whitney stratification of $X$, and then any function $f_{t}$ has a degenerate (non-Morse) singularity at 0, except for very special cases of $X$, or (2). it does belong, but in such a case a general Morsification $f_{t}$ does not have singularity at $0$ with respect to the canonical stratification $\sS$.

As we have already claimed above, the set of Morse pairs $M(f)$ is an extension,  to the stratified setting, of the classical Milnor number. Indeed, in case $X$ is a smooth space there is one single stratum, so $M(f)$ consists of a single pair $(X, m_{X})$ where the integer $m_{X}$  is the classical Milnor number $\mu(f)$.

By \cite[Proposition 2.3]{STV2}, the following formula computes the Morse numbers in terms of the \emph{relative Euler obstruction}\footnote{Defined in \cite{BMPS}.}: 

\begin{proposition}\label{p:MorseSTV} \ 
If $f \colon (X,0) \to (\bC, 0)$ has an isolated stratified singularity, then the following equality holds for any $V\in \sS$ with $\dim V >0$:
\begin{equation}\label{eq:MorseSTV}
  m_{V} = (-1)^{\dim V} \Eu_{f}(\overline{V},0).
\end{equation}

\end{proposition}
\begin{proof}
 For the smooth stratum $V = X_{\reg}$, this formula was stated and proved in \cite[Proposition 2.3]{STV2}. But actually this is all that we need. Indeed,  the closure of any positive dimensional stratum $V$ of the canonical Whitney stratification of $X$ at 0 is a complex analytic space,  hence the same proof (and thus the same formula) applies to the complex analytic space germ $(\overline{V}, 0)$,  since we clearly have $(\overline{V})_{\reg}=V$. 
\end{proof}

Very recently, the equality \eqref{eq:MorseSTV}, in its dual expression for the Euler obstruction\footnote{See \cite{EG} for this definition of the Euler obstruction of the 1-form $\d f$.}, 
was used by Zach \cite{Za} for showing that the Morse numbers can be computed,  in principle, with \v{C}ech cohomology and spectral sequences.

Let us point out that in the global setting of a polynomial function $f\colon X\to \bC$,  there has been growing interest in recent years for computing the number of Morse points especially on the regular part $X_{\reg}$ in a linear Morsification $f_t=f-tl$ of $f$, since this is related to a corresponding notion of an {\it algebraic degree of optimization}, see, e.g., \cite{DHOST}, \cite{MRW2018}, \cite{MRW5} or \cite{MT}, which focus around the \emph{Euclidean distance degree}.

%Our note gives an effective answer in the local setting and for $f$ with isolated stratified singularity.

%%%%%%%%%%%%%%
\section{The polar  formula for the Morse numbers}\label{a:massey}

We will show that a practical and effective formula for the Morse numbers is available in terms of polar multiplicities.
This formula, while proved long ago by Massey in  \cite{Ma},  seems to have been neglected.  It is true that Massey's proof looks long  and tedious, but the formula itself  (see \eqref{eq:main1}) is certainly the most effectively computable. We will give here a much shorter proof with the hope of shedding new light on this beautifully simple formula.

\

We use  the general polar curve defined by Lemma \ref{l:genpolar}.  Whenever  $l\in \Omega'$, if the polar curve $\Gamma (\ell,f)$ is non-empty, then it decomposes as: 
$$\Gamma (\ell,f) = \bigcup_{V\in \sS}\Gamma_{V} (\ell,f)$$
where  $\Gamma_{V} (\ell,f)$ denotes the union of the non-trivial curve components of $\Gamma (\ell,f)$ that are included in $V\cup \{0\}$.

 The following theorem shows that the Morse pairs of $f$ defined above, which enter in the direct sum decomposition of the vanishing homology \eqref{eq:decomp2},  are computable in terms of the local polar multiplicities. In particular this is another proof of the independency  of the set $M(f)$ on the chosen Morsification of $f$.
 %%%%%%%%%%%%%%%%%
 %%%%%%%%%%%%
\begin{theorem}\label{t:main1}
Let $f\colon (X,0) \to (\bC,0)$ be such that $\dim_{0} X \ge 1$ and $f$ has an isolated singularity with respect to the stratification $\sS$.
Then, for any $l\in \Omega'$ and any positive dimensional stratum $V\in \sS$,  we have the equality:
\begin{equation}\label{eq:main1}
 m_{V} = \mult_{0}(\Gamma_{V}(l,f), f^{-1}(0)) - \mult_{0}(\Gamma_{V}(l,f), l^{-1}(0)).
  \end{equation}
\end{theorem}

Our  short proof  of Theorem \ref{t:main1} is a  consequence of two presentations of the homology of the Milnor fibre $F_{f}$ of $f$ at $0$ as a direct sum of certain ``pieces'', the first of which has been recalled in \eqref{eq:decomp2}.   Such a short direct proof  was  also desired by Massey  \cite{Ma} since,  according to his own comment in \cite[p. 1001]{Ma},  he gave a proof of \eqref{eq:main1} in an ``extremely roundabout way''.

\

The second decomposition of the reduced homology of the Milnor fibre $F_{f}$ of $f$ at $0$ is a homological consequence of the Bouquet Theorem  and its proof in \cite{Ti-bouquet}, namely for $\dim X\ge 2$ one has:
\begin{equation}\label{eq:decomp3}
  \widetilde  H_{*}(F_{f}) \cong \widetilde  H_{*}(\lk_X\{0\}) 
  \oplus
   \bigoplus_{V\in \sS\setminus \{0\}}   
   \widetilde H_{*}\left(\Sigma^{\dim V}(\lk_X V) \right)^{k_{V}},
\end{equation}
where:
\begin{equation}\label{eq:polar0}
k_{V} := \mult_{0}(\Gamma_{V}(l,f), f^{-1}(0)) - \mult_{0}(\Gamma_{V}(l,f), l^{-1}(0)),
\end{equation}
as shown  in \cite[\S 4.1]{Ti-bouquet}.

This  direct sum decomposition of the homology in which the exponents are precisely the polar multiplicities $k_{V}$ has been proved in \cite{Ti-bouquet} at the homotopy type level, which demands a much deeper study of the geometry of the monodromy of $f$. What Massey adds up to the (co)-homological picture in \cite{Ma} is the interpretation of the exponents $m_{V}$ in terms of the polar multiplicities $k_{V}$, namely the equality $m_{V}=k_{V}$, cf \cite[Theorem 3.2]{Ma}.

\begin{remark}\label{r:lessgeneral}
The proof of \eqref{eq:decomp3} and \eqref{eq:polar0} in \cite{Ti-bouquet}  uses a very general $l\in \Omega$. However the proof works for any $l\in \Omega'$, with minor changes taking into account the fact that $(l,f)_{| \Gamma (l,f)}$ is finite-to-one
and not one-to-one. The reason to consider $\Omega'$ instead of $\Omega$ is that the computation of intersection numbers may be easier
when working with $l\in \Omega'$, as we will see in Example \ref{ex:1}.
\end{remark}

 %%%%%%%%%%%%%%%%
\subsection{The Comparison Argument}\label{ss:comp}

 Let us first remark that the integers $m_{V}$ and $k_{V}$ do not depend on the chosen $l$ inside this Zariski-open set $\Omega'_{f}$, due to the arc-connectedness of this space.
We prove the equality $m_{V} = k_{V}$  for all positive dimensional strata $V \in \sS$  by induction on the dimension 
$n=\dim_{0} X \ge 2$ and on the number of strata of $\sS$ at $0$.  The case $\dim_{0} X \ge 1$ will be treated separately.
 
 \
 
Let $X$ have all irreducible components of dimension $\ge 2$. We first remark that if  $(X,0)$  is irreducible and $X\setminus \{0\}$ is non-singular in a neighborhood of $0$, then we have $X_{\reg} := X\setminus \{0\}$ as the single stratum outside $0$ in its neighborhood, and therefore the equality of numbers $m_{X_{\reg}} = k_{X_{\reg}}$ is clear in this case by comparing the two decompositions \eqref{eq:decomp2} and \eqref{eq:decomp3}. Let us also  point out that the complex link of the top stratum $X_{\reg}$ is the empty set, and so for $V= X_{\reg}$ we get $\widetilde H_{*}(\Sigma^{\dim X}(\lk_X X_{\reg}) \cong  \widetilde H_{*}(S^{\dim X -1})$.

We start the induction from the case $\dim_{0} X = 0$. 
 Then the statement of Theorem \ref{t:main1} is empty, so there is nothing to prove.
 Suppose that we have proved the equality for all space germs $(X,0)$ of dimension $\le n-1$, and we consider now a space $X$ with $\dim_{0} X = n$.  We consider the union $X_{n-1}$ of the strata at $0$ of dimension $\le n-1$. By the induction hypothesis, the equality of numbers for $X_{n-1}$ is proved stratwise. Then we take the union $X_{n-1}\cup V$ with one of the $n$-dimensional strata $V$ of $X$. Since $X_{n-1}\cup V$ is a closed analytic space germ at 0, we may write for it the two direct sum decompositions \eqref{eq:decomp2} and \eqref{eq:decomp3}, and we compare them. By cancelling out the contributions of $X_{n-1}$ (which are equal by the induction hypothesis), it follows that the equality $m_{V} = k_{V}$ must hold. 
  We then repeat the above described procedure by adding up one by one the other strata of dimension $n$. 
 The equality of numbers for each such stratum follows in the same manner.
 
 Let us observe that the above argument, i.e. comparing the direct sum decompositions \eqref{eq:decomp2} and \eqref{eq:decomp3},  also applies in the excepted case $\dim X = 1$, simply because the two decompositions still hold for the reduced homology in degree 0, namely:
 \[\widetilde H_{0}(F_{f}) \cong  \widetilde H_{0}(\lk_{X}\{0\}) 
  \oplus  \bigoplus_{V\in \mathscr S\setminus \{0\}} \widetilde H_{0}\left(\Sigma^{\dim V}(\lk_{X}V) \right)^{m_{V}}, 
  \]
  respectively:
\[
  \widetilde  H_{0}(F_{f}) \cong \widetilde  H_{0}(\lk_X\{0\}) 
  \oplus
   \bigoplus_{V\in \sS\setminus \{0\}}   
   \widetilde H_{0}\left(\Sigma^{\dim V}(\lk_X V) \right)^{k_{V}},
   \]
and remark that  in this case $\widetilde H_{0}\left(\Sigma^{\dim V}(\lk_X V) \right) \simeq \bZ$, and that $k_{V}$ still verifies the ``polar'' equality \eqref{eq:polar0}.

Thus, by comparing these two decompositions of $\widetilde  H_{0}(F_{f})$, we still get the equality  $m_{V}=k_{V}$.
 In this way we have also treated separately the case of a space $X$ of dimension 1, as promised in the very beginning.
 We have thus reached the end of our relatively short proof of Theorem \ref{t:main1}.
 \fin

\begin{remark}\label{r:morseindep}
The above Comparison Argument can be also applied to the direct sum decompositions of the homology of the Milnor fibre of $f$ coming from two different Morsifications of $f$, with Morse numbers $m_{V}$ and $m'_{V}$, respectively. Then the above inductive procedure shows directly
the equality $m_{V} = m'_{V}$ for each positive dimensional stratum $V\in \sS$, and consequently the independence 
of the set of  Morse pairs $M(f)$.
\end{remark}

\smallskip

%%%%%%%%%%%%%%%

 Let us  compute the set of Morse pairs $M(f)$ for the example considered in \cite{Za}. 
\begin{example}\label{ex:1}
 Let $X := \{ h = xz^{2} - y^{2} =0 \} \subset \bC^{3}$ be the well-known Whitney umbrella.
 One has two positive dimensional strata in the canonical stratification of $X$, namely the $x$-axis without the origin, denoted in the following by $V$, and the 2-dimensional stratum $W := X \setminus \overline{V}$. 
 Let $f := y^{2} -(x-z)^{2}$, viewed as function on $X$. We will consider the germs at $0\in \bC^{3}$ of all the involved objects.  We choose the linear function $l = x+2z$, as in \cite{Za}. %Let us remark that $l\not
 
 Since $V$ has dimension 1, the polar locus $\Gamma_{V}(l,f)$ is equal to $\overline{V}$ by definition. We have $f_{|V}= -x^{2}$ and $l_{|V}= x$,
 hence we get:
 \[\mult_{0}(\Gamma_{V}(l,f), f^{-1}(0)) - \mult_{0}(\Gamma_{V}(l,f), l^{-1}(0)) = \ord_{0}(-x^{2}) - \ord_{0}(x) = 2-1=1,
 \]
thus,  by Theorem \ref{t:main1}   we obtain $m_{V} = 1$.
 
 The polar locus $\Gamma_{W}(l,f)$ is found by first writing down the equations of the singular locus of $f$ on $X$, which is the Jacobian of the triple $(h,f, l)$, and then restricting it to $W$.  
 By computation  we find that this Jacobian is $4y(x-z)(3+z) =0$. It then follows that  $\Gamma_{W}(l,f)$ has the  two branches passing through the origin, namely: $\Gamma_{1} = \{ x=y=0\}$ and $\Gamma_{2} = \{ x=z, y^{2}= xz^{2}\}$.  
 
 At this moment we may
 remark that the linear function $l$ is not very general,  since it does not satisfy the last requirement 
 of Lemma \ref{l:genpolar}. Indeed, the restriction of $(l,f)$ to $\Gamma_{2}$ is two-to-one and not one-to-one. Nevertheless  $l\in \Omega'$, and this is just enough for our setting, see also  Remark \ref{r:lessgeneral}.
 
 We have $\Gamma_{2} = \{ x=z, y^{2}= xz^{2}\} = \{ x=z, y^{2}= x^{3}\}$, and 
by parametrising the cusp $y^{2}= x^{3}$  by $x=t^{2}, y=t^{3}$, we get $l_{|\Gamma_{2}} = 3t^{2}$ and 
$f_{|\Gamma_{2}} = t^{6}$. Thus:
 \[\mult_{0}(\Gamma_{2}, f^{-1}(0)) - \mult_{0}(\Gamma_{2}(l,f), l^{-1}(0)) = \ord_{0}(t^{6}) - \ord_{0}(3t^{2}) = 6-2=4.
 \]
 % Reamark the facility of using the order $\ord$ when we have the parametrisation of Gamma !!
The computation for $\Gamma_{1}$ is much easier since this polar branch is the $z$-axis:
 \[\mult_{0}(\Gamma_{1}, f^{-1}(0)) - \mult_{0}(\Gamma_{1}(l,f), l^{-1}(0)) = \ord_{0}(-z^{2}) - \ord_{0}(2z) = 2-1=1.
 \]
   Summing up the numbers over the two branches we get $4+1 =5$, which is our Morse number $m_{W}$, according to Theorem \ref{t:main1}.

We thus obtain the set of Morse pairs $M(f) = \{ (V,1), (W,5)\}$.  
\end{example}

\begin{remark}
In \cite[pages 937-939]{Za}, Zach explains how to compute  with his formula \cite[Thm. 1 and Cor. 4]{Za} the number $m_{W}=5$ for the top stratum $W$ in the above example,  with the aid of a computer 
using an algorithm implemented on the Singular software. By our method, the computation of $m_{W}$ is easily doable by hand.  However this facility is due to the nice factorisation of the Jacobian in this particular example. In general, it is much more difficult to compute even with our method, therefore an implementation of its algorithm on Singular could be of interest.
  \end{remark}

%%%%%%%%%%%%%%%%

%%%%%%%%%%%%%
%%%%%%%%%%%
\bigskip

\end{document}